\documentclass[12pt]{amsart}
    \usepackage[OT2,T1]{fontenc}
    \usepackage{amsmath,amsfonts,amsthm,amssymb,mathrsfs,
    amsxtra,amscd,latexsym,color,mathdots,verbatim}
    \usepackage[hmargin=2cm,vmargin=3cm]{geometry}
    \DeclareSymbolFont{cyrletters}{OT2}{wncyr}{m}{n}
    \DeclareMathSymbol{\Sha}{\mathalpha}{cyrletters}{"58}

     \newtheorem{thm}{Theorem}[section]

     \newtheorem{lem}[thm]{Lemma}

     \newtheorem{dfn}[thm]{Definition}

     \theoremstyle{definition}
     \theoremstyle{remark}
     \theoremstyle{question}
     \numberwithin{equation}{section}

    \newcommand{\sm}{\left(\begin{smallmatrix}}
    \newcommand{\esm}{\end{smallmatrix}\right)}
    \newcommand{\mat}{\left(\begin{matrix}}
    \newcommand{\emat}{\end{matrix}\right)}

    \newcommand{\mbf}{\mathbf}

    \def\CC{\mathbb{C}}
    \def\HH{\mathbb{H}}

    \def\GL{\mathrm{GL}}

\begin{document}

 \title{$L$-series for Vector-Valued Weakly Holomorphic Modular Forms and Converse Theorems}

   \author{Subong Lim}
   \address{Department of Mathematics Education, Sungkyunkwan University, Jongno-gu, Seoul 110-745, Republic of Korea}
   \email{subong@skku.edu}
   \author{Wissam Raji}
   \address{Department of Mathematics, American University of Beirut (AUB) and the Number Theory Research Unit at the Center for Advanced Mathematical Sciences (CAMS) at AUB, Beirut, Lebanon}
   \email{wr07@aub.edu.lb}

\begin{abstract}
We introduce the $L$-series of weakly holomorphic modular forms using Laplace transforms and give their functional equations.  We then determine converse theorems for vector-valued harmonic weak Maass forms, Jacobi forms, and elliptic modular forms of half-integer weight in Kohnen plus space.
\end{abstract} 
\maketitle

\section{Introduction}
The theory of $L$-functions exhibits natural connections with various mathematical subjects including number fields, automorphic forms, Artin representations, Shimura varieties, abelian varieties, and intersection theory. $L$-series of holomorphic modular forms and Maass forms has been studied extensively and many interesting results have emerged.  However, when it comes to harmonic Maass forms, the study of the forms themselves seems to have attracted more attention than the study of their $L$-series.  There have been several definitions of $L$-series associated with weakly holomorphic modular forms in the literature but almost all have fallen short of recovering a converse theorem except for the definition given in \cite{DLRR}.  Recently, in \cite{DLRR}, using Laplace transforms, $L$-series of harmonic weak Maass forms were introduced. The formulation presented made it possible to present a converse theorem for weakly holomorphic modular forms.\\
\par On the other hand, vector-valued modular forms are important generalizations of elliptic modular forms that arise naturally in the theory of Jacobi forms, Siegel modular forms, and Moonshine.  Vector-valued modular forms have been used as an important tool in tackling classical problems in the theory of modular forms.  For example, Selberg used these forms to give estimates for the Fourier coefficients of the classical modular forms \cite{S}. Borcherds in \cite{B1} and \cite{B2} used vector-valued modular forms associated with Weil representations to describe the Fourier expansion of various theta liftings. Some applications of vector-valued modular forms stand out in high-energy physics by mainly providing a method of differential equations to construct the modular multiplets, and also revealing the simple structure of the modular invariant mass models \cite{DL}. Other applications concerning vector-valued modular forms of half-integer weight seem to provide a simple solution to the Riemann-Hilbert problem for representations of the modular group \cite{BG}. So it is only natural to study the $L$-series of vector-valued modular forms and their properties as a buildup that aligns with the development of a Hecke theory to the space of vector-valued modular forms.\\

\par In this paper, we define $L$-series for vector-valued harmonic weak Maass forms and in particular vector-valued weakly holomorphic modular forms using the Laplace transform where our definition enables us to give a converse theorem.  The definition of our $L$-series is analogous to that in \cite{DLRR}.  The main point of defining our $L$-series is meant to recover a converse theorem for vector-valued weakly holomorphic modular forms.  Moreover, we determine similar definitions for Jacobi forms and modular forms of half-integral weight in Kohnen plus space and determine converse theorems in those cases as well. Converse theorems have historically offered a means to describe Dirichlet series linked to modular forms by examining their analytic properties.  Hecke initially demonstrated that Dirichlet series connected to modular forms exhibit certain analytic properties, and then went on to show conversely that these analytic properties define modular Dirichlet series \cite{E}.  A common application of a Converse Theorem that might initially come to mind is in addressing the modularity of certain arithmetic or geometric objects associated with a given $L$-series. Instead, the most natural way of applying the Converse theorem is to Functoriality. For example, the transfer of automorphic representations from some group G to $GL_n$. In another application of some of these cases of Functoriality, Kim and Shahidi have established the best-known general bounds towards the Ramanujan conjecture for $GL_2$ \cite{KS}.  As a result, the main focus of our paper is to recover converse theorems for different modular objects in the vector-valued case through their $L$-series that are defined using Laplace transforms.

\section{$L$-series of vector-valued weakly holomorphic modular forms} \label{Lweakly}
This section introduces the $L$-series of a vector-valued weakly holomorphic modular form and proves some of its properties. In particular, we prove a converse theorem.  We start by introducing some notation.\\

\par
  Let $\Gamma = \mathrm{SL}_2(\mathbb{Z})$. 
 Let $k\in \frac12 \mathbb{Z}$ and $\chi$ a unitary multiplier system of weight $k$ on $\Gamma$, i.e. $\chi:\Gamma\to\mathbb{C}$ satisfies the following conditions:
  \begin{enumerate}
  \item $|\chi(\gamma)| = 1$ for all $\gamma\in \Gamma$.
  \item $\chi$ satisfies the consistency condition
  \[
  \chi(\gamma_3) (c_3\tau + d_3)^k = \chi(\gamma_1)\chi(\gamma_2) (c_1\gamma_2\tau + d_1)^k (c_2\tau+d_2)^k,
  \]
  where $\gamma_3 = \gamma_1\gamma_2$ and $\gamma_i = \sm a_i&b_i\\c_i&d_i\esm\in \Gamma$ for $i=1,2$, and $3$.
  \end{enumerate}
  Let $m$ be a positive integer and $\rho:\Gamma\to \mathrm{GL}(m, \mathbb{C})$ a $m$-dimensional unitary complex representation. 
  Let $\{\mbf{e}_1, \ldots, \mbf{e}_m\}$ denote the standard basis of $\mathbb{C}^m$. 
  For a vector-valued function $f = \sum_{j=1}^m f_j \mbf{e}_j$ on $\mathbb{H}$ and $\gamma\in \Gamma$, define a slash operator by
  \[
  (f|_{k,\chi,\rho}\gamma)(\tau):= (c\tau+d)^{-k} \chi^{-1}(\gamma) \rho^{-1}(\gamma) f(\gamma \tau).
  \]
  The definition of the vector-valued modular forms is given as follows.
  
  \begin{dfn}
  A vector-valued weakly holomorphic  modular form of weight $k$, multiplier system $\chi$, and type $\rho$ on $\Gamma$ is a sum $f = \sum_{j=1}^m f_j \mbf{e}_j$ of functions holomorphic in  $\mathbb{H}$ satisfying the following conditions:
  \begin{enumerate}
  \item $f|_{k,\chi,\rho}\gamma = f$ for all $\gamma\in \Gamma$.
  \item For each $1\leq j\leq m$, each function $f_j$ has a Fourier expansion of the form
  \begin{equation} \label{vvmfFourier}
  f_i(\tau) = \sum_{n+\kappa_j\gg 0} a_{f,j}(n)e^{2\pi i(n+\kappa_j)\tau}.
  \end{equation}
  Here and throughout the paper, $\kappa_j$ is a certain positive number with $0\leq \kappa_j<1$. 
  \end{enumerate}
  \end{dfn}
  
  The space of all vector-valued weakly holomorphic modular forms of weight $k$, multiplier system $\chi$, and type $\rho$ on $\Gamma$ is denoted by $M^!_{k,\chi,\rho}$. 
  There is a subspace $S_{k,\chi,\rho}$ of vector-valued cusp forms for which we require each $a_j(n) = 0$ when $n+\kappa_j$ is non-positive.
  
  For a vector-valued cusp form $f(\tau) = \sum_{j=1}^m \sum_{n+\kappa_j>0} a_j(n)e^{2\pi i(n+\kappa_j)\tau}\mbf{e}_j\in S_{k,\chi,\rho}$, then $a_j(n) = O(n^{k/2})$ for every $1\leq j\leq m$ as $n\to\infty$ by the same argument for classical modular forms (for more details, see \cite{KM}). 
  Then, the vector-valued $L$-series defined by
  \begin{equation} \label{usual}
  L(f,s) := \sum_{j=1}^m \sum_{n+\kappa_j>0} \frac{a_j(n)}{(n+\kappa_j)^s}\mbf{e}_j
  \end{equation}
  converges absolutely for $\mathrm{Re}(s)\gg0$. 
   This has an integral representation
  \[
  \frac{\Gamma(s)}{(2\pi)^s} L(f,s) = \int_0^\infty f(iv) v^s \frac{dv}{v}.
  \]  
  From this, we see that it has an analytic continuation to $\mathbb{C}$ and a functional equation 
 \[
 L^*(f,s) = i^k \chi(S) \rho(S) L^*(f,k-s),
 \]
 where $L^*(f,s) = \frac{\Gamma(s)}{(2\pi)^s} L(f,s)$ and $S = \sm 0&-1\\1&0\esm$ (for example, see \cite{JL, LR, LR2}).

Let $f$ be a vector-valued weakly holomorphic modular form in $M^!_{k,\chi,\rho}$ with Fourier expansion in (\ref{vvmfFourier}).
Let $n_0\in\mathbb{N}$ be such that $f_j(\tau)$ are $O(e^{2\pi n_0 v})$ as $v=\mathrm{Im}(\tau)\to\infty$ for each $1\leq j\leq m$.
Let $\mathcal{F}_f$ be the space of test functions $\varphi:\mathbb{R}_+ \to \mathbb{C}$ such that
\begin{enumerate}
\item
 $(\mathcal{L}\varphi)(s)$ converges for all $s$ with $\mathrm{Re}(s) \geq -2\pi n_0$, 

\item the series
\[
\sum_{n\gg-\infty} | a_{f,j}(n)| (\mathcal{L}|\varphi_j|)(2\pi (n+\kappa))
\]
converges for each $1\leq j\leq m$.
\end{enumerate}
The space $\mathcal{F}_f$ contains the compactly supported smooth functions on $\mathbb{R}_+$ because of the growth of $a_{f,j}(n)$ (for more details of the growth of $a_{f,j}(n)$, see \cite{BF}). 
We define the vector-valued $L$-series map $L_f : \mathcal{F}_f \to \mathbb{C}^m$ by
\[
L_f(\varphi) := \sum_{j=1}^m \sum_{n\gg-\infty} a_{f,j}(n) (\mathcal{L} \varphi)(2\pi (n+\kappa_j))\mbf{e}_j.
\]
Then, we prove that this $L$-series has an integral representation and satisfies a functional equation.

\begin{thm} \label{property}
Let $f$ be a vector-valued weakly holomorphic modular form in $M^!_{k,\chi,\rho}$.
\begin{enumerate}
\item For $s\in \mathbb{C}$, let $
I_s(x) :=  (2\pi)^s x^{s-1} \frac{1}{\Gamma(s)}$. 
If $f$ is a vector-valued cusp form, then $L_f(I_s) = L(f,s)$.

\item For $\varphi\in \mathcal{F}_f$, the $L$-series $L_f(\varphi)$ can be given by
\[
L_f(\varphi) = \int_0^\infty f(iy) \varphi(y) dy.
\]

\item We have the following functional equation:
\[
L_f(\varphi) = i^k \rho(S) L_f(\varphi|_{2-k,\chi^{-1}} S),
\]
where $(\varphi|_{2-k,\chi^{-1}} S)(x) := x^{k-2}\chi(S) \varphi\left( \frac1 x\right)$ for all $x>0$.

\end{enumerate}
\end{thm}

\begin{proof}
(1) For $u>0$ and $\mathrm{Re}(s)>0$, we have
\[
(\mathcal{L} I_s)(u) = \frac{(2\pi)^s}{\Gamma(s)} \int_0^\infty e^{-ut} t^{s-1}dt = \left(\frac{2\pi}{u}\right)^s.
\]
Therefore, if $f$ is a vector-valued cusp form, then  $L_f(I_s)$ is the usual $L$-series of $f$ defined in (\ref{usual}).

(2) Suppose that $f$ has a Fourier expansion as in (\ref{vvmfFourier}).
By the definition of $L_f(\varphi)$, we have
\begin{eqnarray*}
L_f(\varphi) &=& \sum_{j=1}^m \sum_{n\gg-\infty} a_{f,j}(n) (\mathcal{L} \varphi_j)(2\pi (n+\kappa_j))\mbf{e}_j\\
&=& \sum_{j=1}^m \sum_{n\gg-\infty} a_{f,j}(n) \int_0^\infty e^{-2\pi (n+\kappa_j)t}\varphi_j(t)dt \mbf{e}_j\\
&=& \sum_{j=1}^m  \left(\int_0^\infty  f_j(it) \varphi_j(t)  dt\right) \mbf{e}_j.
\end{eqnarray*}
The last equality follows from the fact that we can interchange the order of summation and integration since $\varphi\in \mathcal{F}_f$.

(3) Since $f\in M^!_{k,\chi,\rho}$, we see that 
\[
f(iy) = (iy)^{-k} \chi^{-1}(S) \rho(S)^{-1} f\left(i\frac1y\right)
\]
for $y>0$. 
Therefore, we have
\begin{eqnarray*}
L_f(\varphi) &=& \int_0^\infty f(iy)\varphi(y)dy\\
&=& \int_0^\infty f\left( i\frac 1y \right) \varphi \left( \frac 1y \right) y^{-2}dy\\
&=& \int_0^\infty  (iy)^k \chi(S) \rho(S) f(iy) \varphi \left( \frac 1y \right) y^{-2}dy\\
&=& i^k \rho(S) \int_0^\infty f(iy) y^{k-2} \chi(S) \varphi \left( \frac 1y\right) dy.
\end{eqnarray*}
\end{proof}

Let $\mathcal{C}(\mathbb{R}, \mathbb{C})$ be the space of piece-wise smooth complex-valued functions on $\mathbb{R}$.
For $s\in\mathbb{C}$ and $\varphi\in \mathcal{C}(\mathbb{R}, \mathbb{C})$, we define $\varphi_s(x) := \varphi(x) x^{s-1}$.
We also define the series 
\[
L(s,f,\varphi) := L_f(\varphi_s).
\]
Then, we prove that this series has an analytic continuation to all $s\in\mathbb{C}$ and satisfies a functional equation.

\begin{thm} \label{varphi_s}
Let $f\in M^!_{k, \chi,\rho}$, and  $n_0\in\mathbb{N}$ be such that $f_j(\tau)$ are $O(e^{2\pi n_0 v})$ as $v=\mathrm{Im}(\tau)\to\infty$ for each $1\leq j\leq m$. 
Suppose that $\varphi \in \mathcal{C}(\mathbb{R}, \mathbb{C})$ is a non-zero function such that, for some $\epsilon>0$, $\varphi(x)$ and $\varphi(x^{-1})$ are $o(e^{-2\pi(n_0+\epsilon)x})$ as $x\to\infty$. 
 Then the series $L(s,f,\varphi)$ 
converges absolutely for $\mathrm{Re} > \frac12$, has an analytic continuation to all $s\in\mathbb{C}$ and satisfies the functional equation
\[
L(s,f,\varphi) = i^k \rho(S) L(1-s, f, \varphi|_{1-k,\chi^{-1}} S).
\]
\end{thm}

\begin{proof}
By the growth of $\varphi$, we see that $\mathcal{L}(|\varphi|^2)(y)$ converges absolutely for $y\geq -2\pi n_0$. 
For $y>0$ and $s\in\mathbb{C}$ with $\mathrm{Re}(s)>\frac12$, Cauchy-Schwarz inequality implies that
\[
(\mathcal{L}|\varphi_s|)(y) \leq (\mathcal{L}(|\varphi|^2)(y))^{\frac12} y^{-\mathrm{Re}(s)+\frac12} (\Gamma(2\mathrm{Re}(s)-1))^{\frac12}.
\]
Therefore, $\varphi_s\in \mathcal{F}_f$ for $\mathrm{Re}(s)>\frac12$. 

Recall that 
\[
f(iy) = (iy)^{-k} \chi^{-1}(S) \rho(S)^{-1} f\left(i\frac1y\right)
\]
for $y>0$. 
Therefore, we have
\begin{eqnarray*}
L(s, f, \varphi) &=& L_f(\varphi_s) = \int_0^\infty f(iy)\varphi_s(y)dy\\
&=& \int_0^\infty f(iy) \varphi(s) y^{s-1} dy\\
&=& \int_1^\infty f(iy) \varphi(y)y^{s-1} dy + \int_0^1 f(iy)\varphi(y)y^{s-1}dy\\
&=& \int_1^\infty f(iy) \varphi(y)y^{s-1} dy + \int_1^\infty f\left( i\frac1y \right) \varphi \left( \frac 1y \right) y^{1-s}y^{-2}dy\\
&=& \int_1^\infty f(iy) \varphi(y)y^{s-1} dy + i^k \rho(S) \int_1^\infty f(iy) \left(\varphi|_{1-k,\chi^{-1}}S\right)(y) y^{-s}dy.
\end{eqnarray*}
By the growth of $\varphi$ at $0$ and $\infty$, we see that the integrals 
\[
\int_1^\infty f(iy) \varphi(y)y^{s-1} dy\  \text{and}\ \int_1^\infty f(iy) \left(\varphi|_{1-k,\chi^{-1}}S\right)(y) y^{-s}dy
\]
 are well-defined for all $s\in\mathbb{C}$, and give holomorphic functions.

Since $f\in M^!_{k, \chi,\rho}$, we obtain that $\rho(-I_2) \chi(-I_2) = - I_m$, where $I_m$ denotes the identity matrix of size $m$. 
Therefore, we get the desired functional equation. 
\end{proof}

Let $S_c(\mathbb{R}_+)$ be a set of complex-valued, compactly supported, and piecewise smooth functions on $\mathbb{R}_+$, which satisfy the following condition: for any $y\in\mathbb{R}_+$, there exists $\varphi\in S_c(\mathbb{R}_+)$ such that $\varphi(y)\neq 0$. 
We write $\langle \cdot, \cdot\rangle$ for the standard scalar product on $\mathbb{C}^m$, i.e.
\[
\left \langle \sum_{j=1}^m \lambda_j \mbf{e}_j,  \sum_{j=1}^m \mu_j \mbf{e}_j \right \rangle = \sum_{j=1}^m \lambda_j \overline{\mu_j}.
\]
We now state the converse theorem in the case of vector-valued weakly holomorphic modular forms.

\begin{thm} \label{converse}
For each $1\leq j\leq m$, let $(a_{f,j}(n))_{n\geq-n_0}$ be a sequence of complex numbers such that $a(n) = O(e^{C\sqrt{n}})$ as $n\to\infty$, for some $C>0$. 
For each $\tau\in\mathbb{H}$, set
\[
f(\tau) := \sum_{j=1}^m \sum_{n=-n_0}^\infty a_{f,j}(n)e^{2\pi i(n+\kappa_j)\tau} \mbf{e}_j.
\]
Suppose that for each $\varphi\in S_c(\mathbb{R}_+)$, the function $L_f(\varphi)$ satisfies 
\[
L_f(\varphi) = i^k \rho(S) L_f(\varphi|_{2-k, \chi^{-1}} S).
\]
Then, $f$ is a vector-valued weakly holomorphic modular form in $M^!_{k,\chi,\rho}$.
\end{thm}

\begin{proof}
By the bound for $a_{f,j}(n)$, we see that $f_j(\tau)$ converges absolutely to a smooth function on $\mathbb{H}$ for $1\leq j\leq m$. 
Note that $\varphi_s \in S_c(\mathbb{R}_+)$ for any $s\in\mathbb{C}$ and $\varphi\in S_c(\mathbb{R}_+)$.
Since $\varphi\in S_c(\mathbb{R}_+)$, there exist $0<c_1<c_2$ and $C>0$ such that $\mathrm{Supp}(\varphi) \subset [c_1, c_2]$ and $|\varphi(y)|\leq C$ for any $y>0$. 
Then, for each $1\leq j\leq m$ and $n+\kappa_j>0$, we have
\begin{eqnarray*}
|a_{f,j}(n)| (\mathcal{L}|\varphi_s|) (2\pi (n+\kappa_j)) &\leq& C |a_{f,j}(n)| e^{-2\pi (n+\kappa_j)c_1} (c_2-c_1) \max\{c_1^{\mathrm{Re}(s)-1}, c_2^{\mathrm{Re}(s)-1}\}.
\end{eqnarray*}
This implies that for each $1\leq j\leq m$, we have
\begin{eqnarray*}
\sum_{n\in\mathbb{Z} \atop n+\kappa_j>0} |a_{f,j}(n)| (\mathcal{L}|\varphi_s|)(2\pi (n+\kappa_j)) \leq C(c_2-c_1)\max\{c_1^{\mathrm{Re}(s)-1}, c_2^{\mathrm{Re}(s)-1}\} \sum_{n\in\mathbb{Z} \atop n+\kappa_j>0} |a_{f,j}(n)| e^{-2\pi (n+\kappa_j) c_1}.
\end{eqnarray*}
Therefore, $\varphi_s \in \mathcal{F}_f$, and $L_f(\varphi_s)$ is an analytic function on $s\in\mathbb{C}$ by Weierstrass theorem. 

Recall that by Theorem \ref{property} we have
\[
L_f(\varphi_s) = \int_0^\infty f(iy) \varphi_s(y) dy.
\]
By the Mellin inversion formula, we have
\[
f(iy) \varphi(y) = \frac{1}{2\pi i}\int_{c-i\infty}^{c+i\infty} L_f(\varphi_s) y^{-s} ds
\]
for all $c\in \mathbb{R}$ and $y>0$. 
By integration by parts, we see that for each $1\leq j\leq m$, 
\[
|\langle L_f(\varphi_s), \mbf{e}_j\rangle | \leq \frac{1}{|s|} \int_0^\infty \left|\frac{d}{dy} (f_j(iy)\varphi(y))\right| y^{\mathrm{Re}(s)} dy.
\]
Therefore, $L_f(\varphi_s) \to 0$ as $\mathrm{Im}(s)\to \infty$. 
From this, we have
\begin{eqnarray*}
f(iy) \varphi(y) &=& \frac{1}{2\pi i}\int_{c-i\infty}^{c+i\infty} L_f(\varphi_s) y^{-s}ds=\frac{1}{2\pi i} \int_{k-c-i\infty}^{k-c+i\infty} L_f(\varphi_s) y^{-s}ds\\
&=& \frac{1}{2\pi i} \int_{c-i\infty}^{c+i\infty} L_f(\varphi_{k-s}) y^{-k+s} ds,
\end{eqnarray*}
and by Theorem \ref{property}, we see that
\begin{equation} \label{ftneqn1}
f(iy) \varphi(y) = \frac{i^k}{2\pi i} \rho(S) \int_{c-i\infty}^{c+i\infty} L_f(\varphi_{k-s}|_{2-k,\chi^{-1}} S)y^{-k+s}ds
\end{equation}
for any $y>0$.

Since we have
\[
L_f(\varphi_{k-s}|_{2-k} S) = \int_0^\infty f(iy)(\varphi_{k-s}|_{2-k,\chi^{-1}} S)(y)dy = \int_0^\infty f(iy) \chi(S) \varphi\left( \frac 1y \right) y^{s-1} dy,
\]
the Mellin inversion formula implies that
\[
f(iy) \chi(S) \varphi\left( \frac1y \right)  = \frac{1}{2\pi i}\int_{c-i\infty}^{c+i\infty} L_f(\varphi_{k-s}|_{2-k,\chi^{-1}}S)y^{-s}dy
\]
for all $c\in \mathbb{R}$ and $y>0$.
Therefore, we have
\begin{equation} \label{ftneqn2}
f\left( \frac{i}{y} \right) \chi(S) \varphi(y) = \frac{1}{2\pi i} \int_{c-i\infty}^{c+i\infty} L_f(\varphi_{k-s}|_{2-k,\chi^{-1}}S) y^{s} dy,
\end{equation}
and from (\ref{ftneqn1}) and (\ref{ftneqn2}), we see that 
\[
f(iy) \varphi(y) = i^k y^{-k} \rho(S) \chi(S) f\left( \frac iy \right) \varphi(y)
\]
for any $y>0$.
Since for each $y>0$, there exists $\varphi \in S_c(\mathbb{R}_+)$ such $\varphi(y) \neq 0$, we have
\begin{equation} \label{ftneqn3}
 f(iy)= i^k y^{-k} \rho(S) \chi(S) f\left( \frac iy \right)
\end{equation}
for any $y>0$. 
Since $f$ is a holomorphic function, this implies that $f$ satisfies
\[
f\left( -\frac{1}{\tau} \right) = z^k \chi(S) \rho(S) f(\tau)
\]
for $\tau\in\mathbb{H}$. 
Hence, $f$ satisfies
\[
f|_{k, \chi,\rho}T = f\ \text{and}\ f|_{k,\chi,\rho}S = f.
\]
Since $T$ and $S$ generates $\mathrm{SL}_2(\mathbb{Z})$, we see that $f$ is a weakly holomorphic modular form in $M^!_{k,\chi,\rho}$. 
\end{proof}

\section{$L$-series of vector-valued harmonic weak Maass forms}
In this section, we review basic definitions of vector-valued harmonic weak Maass forms and their $L$-series.  We prove that the properties of $L$-series in Section \ref{Lweakly} also hold in the case of vector-valued harmonic weak Maass forms.  Moreover, we prove a converse theorem and a summation formula for vector-valued harmonic weak Maass forms.

Following \cite{BF, JL2}, we introduce vector-valued harmonic weak Maass forms. 
\begin{dfn} \label{HMF}
A vector-valued harmonic Maass form of weight $k$, multiplier system $\chi$, and type $\rho$ on $\Gamma$ is a real-analytic vector-valued function $f=\sum_{j=1}^{m} f_j\mbf{e}_j$ on $\mathbb{H}$ 
that satisfies the following conditions:
\begin{enumerate}
\item $f|_{k,\chi,\rho}\gamma=f$ for all $\gamma=\sm a&b\\c&d\esm\in \Gamma$.

\item\label{Lap} $\Delta_k f=0$, where $\Delta_k:=-4v^2 \frac{\partial}{\partial \tau} \frac{\partial} {\partial{\overline{\tau}}}+2kiv\frac{\partial}{\partial{\overline{\tau}}}$ is the weight $k$ hyperbolic Laplacian and $\tau=u+iv\in\mathbb{H}$. 

\item It has a Fourier expansion of the form
\begin{align} \label{Fourierharmonic}
f(\tau)&=
\sum_{j=1}^{m}\sum_{n\gg -\infty}c^+_{f,j}(n) e^{2\pi i (n+\kappa_{j})\tau}\mbf{e}_j\\
&\nonumber\quad+ \sum_{j=1}^{m}
\sum_{n+\kappa_j< 0}c^-_{f,j}(n)\Gamma\left(1-k,-4\pi (n+\kappa_{j})v \right) e^{2\pi i (n+\kappa_{j})\tau}\mbf{e}_j,
\end{align}
where $\Gamma(k,w)$ is the analytic continuation of the incomplete gamma function given by $\int_{w}^\infty e^{-t}t^{k-1}dt$.
\end{enumerate}
\end{dfn}

We use $H_{k,\chi,\rho}$ to denote the space of such forms.  
We write $f^{+}$ (resp. $f^-$) for the first (resp. second) summation of (\ref{Fourierharmonic}) and call it the holomorphic (resp. non-holomorphic) part of $f$.
By \cite[Lemma 3.4]{BF}, if $f\in H_{k,\chi,\rho}$ with its Fourier expansion as in (\ref{Fourierharmonic}), then there is a constant $C>0$ such that the Fourier coefficients satisfy
\begin{eqnarray*}
c^+_{f,j}(n) &=& O(e^{C\sqrt{n}}),\ n\to +\infty,\\
c^-_{f,j}(n) &=& O(|n|^{k/2}),\ n\to-\infty.
\end{eqnarray*}

Let $\mathcal{C}(\mathbb{R}, \mathbb{C})$ be the space of piecewise smooth complex-valued functions on $\mathbb{R}$.
Suppose that $k\in\frac12\mathbb{Z}$ and 
 $f:\mathbb{H}\to\mathbb{C}^m$ is a vector-valued function on  $\mathbb{H}$ given by the absolutely convergent series as in (\ref{Fourierharmonic}). 
Let $n_0\in\mathbb{N}$ be such that $f_j(\tau)$ are $O(e^{2\pi n_0 v})$ as $v=\mathrm{Im}(\tau)\to\infty$ for each $1\leq j\leq m$.
Let $\mathcal{F}_f$ be the space of functions $\varphi\in C(\mathbb{R}, \mathbb{C})$ such that 
\begin{enumerate}
\item $(\mathcal{L} \varphi)(s)$ converges absolutely for all $s$ with $\mathrm{Re}(s) \geq -2\pi n_0$,

\item $(\mathcal \varphi_{2-k})(s)$ converges absolutely for all $s$ with $\mathrm{Re}(s)>0$,

\item for each $1\leq j\leq m$, the series
\begin{eqnarray*}
&& \sum_{n\gg-\infty} |c^+_{f,j}(n)| (\mathcal{L} |\varphi|)(2\pi (n+\kappa_j))\\
&& + \sum_{n+\kappa_j<0} |c^-_{f,j}(n)| (4\pi |n+\kappa_j|)^{1-k} \int_0^\infty \frac{(\mathcal{L}|\varphi_{2-k}|)(-2\pi (n+\kappa_j)(2t+1))}{(1+t)^k} dt
\end{eqnarray*}
converges.
\end{enumerate}
For $\varphi\in \mathcal{F}_f$, we define the $L$-series of $f$ by
\begin{eqnarray} \label{vvL}
\nonumber L_f(\varphi) &:=& \sum_{j=1}^m \sum_{n\gg-\infty} c^+_{f,j}(n) (\mathcal{L}\varphi) (2\pi (n+\kappa_j))\ \mbf{e}_j\\
&& + \sum_{j=1}^m \sum_{n+\kappa_j<0} c^-_{f,j}(n)(-4\pi (n+\kappa_j))^{1-k} \int_0^\infty \frac{(\mathcal{L}\varphi_{2-k})(-2\pi (n+\kappa_j)(2t+1))}{(1+t)^k} dt\ \mbf{e}_j.
\end{eqnarray}

We also define the following derivative
\[
(\delta_k f)(\tau) := \tau \frac{\partial f}{\partial u}(\tau) + \frac{k}{2} f(\tau).
\]
Then, we have
\begin{eqnarray*}
(\delta_k f)(\tau) &=& \frac{k}{2} f(\tau) + \sum_{j=1}^m \sum_{n\gg-\infty} c^+_{f,j}(n) (2\pi i(n+\kappa_j)\tau)e^{2\pi i(n+\kappa_j)\tau}\mbf{e}_j\\
&& + \sum_{j=1}^m \sum_{n+\kappa_j<0} c^-_{f,j}(n) (2\pi i(n+\kappa_j))\Gamma(1-k, -4\pi (n+\kappa_j)v) e^{2\pi i(n+\kappa_j)\tau}\mbf{e}_j.
\end{eqnarray*}
Let $\mathcal{F}_{\delta_k f}$ be the space of functions $\varphi\in C(\mathbb{R}, \mathbb{C})$ such that 
\begin{enumerate}
\item $(\mathcal{L} \varphi)(s)$ converges absolutely for all $s$ with $\mathrm{Re}(s) \geq -2\pi n_0$,

\item $(\mathcal \varphi_{3-k})(s)$ converges absolutely for all $s$ with $\mathrm{Re}(s)>0$,

\item
for each $1\leq j\leq m$ the series
\begin{eqnarray*}
&& \sum_{n\gg-\infty} |c^+_{f,j}(n)(n+\kappa_j)| (\mathcal{L} |\varphi_2|)(2\pi (n+\kappa_j))\\
&& + \sum_{n+\kappa_j<0} |c^-_{f,j}(n)(n+\kappa_j)| (4\pi |n+\kappa_j|)^{1-k} \int_0^\infty \frac{(\mathcal{L}|\varphi_{3-k}|)(-2\pi (n+\kappa_j)(2t+1))}{(1+t)^k} dt
\end{eqnarray*} 
converges.
\end{enumerate}
For $\varphi\in \mathcal{F}_{\delta_k f}$, we define $L_{\delta_k f}(\varphi)$ by
\begin{eqnarray} \label{vvLdelta}
\nonumber L_{\delta_k f}(\varphi) &:=& \frac{k}{2}L_f(\varphi) -2\pi  \sum_{j=1}^m \sum_{n\gg-\infty} c^+_{f,j}(n) (n+\kappa_j) (\mathcal{L}\varphi_2) (2\pi (n+\kappa_j))\ \mbf{e}_j\\
&&\qquad   -2\pi \sum_{j=1}^m \sum_{n+\kappa_j<0} c^-_{f,j}(n)(n+\kappa_j)(-4\pi (n+\kappa_j))^{1-k}\\
\nonumber &&\qquad \qquad \times \int_0^\infty \frac{(\mathcal{L}\varphi_{3-k})(-2\pi (n+\kappa_j)(2t+1))}{(1+t)^k} dt\ \mbf{e}_j.
\end{eqnarray}
Then, the series $L_f(\varphi)$ and $L_{\delta_k f}(\varphi)$ have integral representations.

\begin{thm} \label{integralrepresentation}
Let $f:\mathbb{H}\to\mathbb{C}^m$ be a vector-valued function on $\mathbb{H}$ as a series in (\ref{Fourierharmonic}). 
\begin{enumerate}
\item For $\varphi\in \mathcal{F}_f$, the $L$-series $L_f(\varphi)$ can be given by
\[
L_f(\varphi) = \int_0^\infty f(iy)\varphi(y)dy.
\]

\item For $\varphi\in \mathcal{F}_{\delta_k f}$, we have
\[
L_{\delta_k f}(\varphi) = \int_0^\infty (\delta_k f)(iy) \varphi(y) dy.
\]
\end{enumerate}
\end{thm}

\begin{proof}
We only prove (1) since the same proof works for (2).
 For the holomorphic part of $f$, we have
\begin{eqnarray*}
&&\sum_{j=1}^m \sum_{n\gg-\infty} c^+_{f,j}(n) (\mathcal{L} \varphi_j)(2\pi (n+\kappa_j))\mbf{e}_j\\
&&= \sum_{j=1}^m \sum_{n\gg-\infty} c^+_{f,j}(n) \int_0^\infty e^{-2\pi (n+\kappa_j)t}\varphi_j(t)dt \mbf{e}_j\\
&&= \sum_{j=1}^m  \left(\int_0^\infty  f^+_j(it) \varphi_j(t)  dt\right) \mbf{e}_j.
\end{eqnarray*}
The last equality follows from the fact that we can interchange the order of summation and integration since $\varphi\in \mathcal{F}_f$.

For the nonholomorphic part of $f$, note that
\[
\Gamma(a, z) = z^a e^{-z} \int_0^\infty \frac{e^{-zt}}{(1+t)^{1-a}}dt,
\]
is valid for $\mathrm{Re}(z)>0$ \cite[(8.6.5)]{OLBC}.
Therefore, for $n+\kappa_j<0$, we have
\begin{eqnarray*}
&&\int_0^\infty \Gamma(1-k, -4\pi(n+\kappa_j) y) e^{-2\pi (n+\kappa_j) u} \varphi(y) dy \\
&&=\int_0^\infty (-4\pi(n+\kappa_j) y)^{1-k} e^{4\pi (n+\kappa_j)y} e^{-2\pi (n+\kappa_j) y} \varphi(y) \int_0^\infty \frac{e^{4\pi(n+\kappa_j)yt}}{(1+t)^k} dtdy\\
&&= (-4\pi(n+\kappa_j))^{1-k}
\int_0^\infty \int_0^\infty y^{1-k}e^{2\pi(n+\kappa_j)y(2t+1)}\varphi(y)dy \frac{1}{(1+t)^k}dt\\
&& = (-4\pi(n+\kappa_j))^{1-k} \int_0^\infty \frac{(\mathcal{L}\varphi_{2-k})(-2\pi (n+\kappa_j)(2t+1))}{(1+t)^k} dt.
\end{eqnarray*}
Therefore, we obtain
\begin{eqnarray*}
&& \sum_{j=1}^m \sum_{n+\kappa_j<0} c^-_{f,j}(n)(-4\pi (n+\kappa_j))^{1-k} \int_0^\infty \frac{(\mathcal{L}\varphi_{2-k})(-2\pi (n+\kappa_j)(2t+1))}{(1+t)^k} dt\ \mbf{e}_j\\
&&= \sum_{j=1}^m \sum_{n+\kappa_j<0} c^-_{f,j}(n) \int_0^\infty \Gamma(1-k, -4\pi(n+\kappa_j)y) e^{-2\pi (n+\kappa_j)y}\varphi(y)dy \mbf{e}_j\\
&&= \int_0^\infty f^-(iy)\varphi(y)dy.
\end{eqnarray*}
\end{proof}

We now prove the functional equations of $L_f(\varphi)$ and $L_{\delta_k f}(\varphi)$.

\begin{thm} \label{ftnaleqn}
Let $f$ be a vector-valued harmonic weak Maass form in $H_{k,\chi,\rho}$.
Then, we have the following functional equations:
\[
L_f(\varphi) = i^k \rho(S) L_f(\varphi|_{2-k,\chi^{-1}} S)
\]
and
\[
L_{\delta_k f}(\varphi) = -i^k \rho(S) L_{\delta_k f}(\varphi|_{2-k, \chi^{-1}} S).
\]
\end{thm}

\begin{proof}
For the first equality, note that since $f\in H_{k,\chi,\rho}$, we see that 
\[
f(iy) = (iy)^{-k} \chi^{-1}(S) \rho(S)^{-1} f\left(i\frac1y\right)
\]
for $y>0$. 
Then, by Theorem \ref{integralrepresentation}, we have
\begin{eqnarray*}
L_f(\varphi) &=& \int_0^\infty f(iy)\varphi(y)dy\\
&=& \int_0^\infty f\left( i\frac 1y \right) \varphi \left( \frac 1y \right) y^{-2}dy\\
&=& \int_0^\infty  (iy)^k \chi(S) \rho(S) f(iy) \varphi \left( \frac 1y \right) y^{-2}dy\\
&=& i^k \rho(S) \int_0^\infty f(iy) y^{k-2} \chi(S) \varphi \left( \frac 1y\right) dy\\
& =& i^k \rho(S) L_f(\varphi|_{2-k, \chi^{-1}}S).
\end{eqnarray*}

For the second equality, note that
\begin{eqnarray*}
(\delta_k(f|_k S))(\tau) &=& \tau (-k) \tau^{-k-1} \chi^{-1}(S) \rho(S)^{-1} f\left(-\frac1\tau \right) + \tau \tau^{-k}\chi^{-1}(S)\rho(S)^{-1}\frac{\partial f}{\partial u} \left(-\frac 1\tau \right) \frac{1}{\tau^2}\\
&&+ \frac k2 \tau^{-k}\chi^{-1}(S)\rho(S)^{-1} f\left( -\frac 1\tau \right)\\
&=& -\frac k2 \tau^{-k} \chi^{-1}(S) \rho(S)^{-1} f\left( -\frac 1\tau \right) - \tau^{-k}\chi^{-1}(S)\rho(S)^{-1}\frac{\partial f}{\partial u}\left( -\frac 1\tau \right) \frac 1\tau\\
&=& - (\delta_k(f)|_k S)(\tau).
\end{eqnarray*}
Therefore, we have
\[
((\delta_k f)|S)(\tau) = -(\delta_k(f|_k S))(\tau) = -(\delta_k f)(\tau).
\]
From this and Theorem \ref{integralrepresentation}, we see that
\begin{eqnarray*}
L_{\delta_k f}(\varphi) &=& \int_0^\infty (\delta_k f)(iy) \varphi(iy) dy\\
&=& \int_0^\infty (\delta_k f)\left( i\frac 1y \right) \varphi \left( \frac1y \right) y^{-2} dy\\
&=& i^k \rho(S) \int_0^\infty (\delta_k f|_k S)(iy) (\varphi|_{2-k,\chi^{-1}}S)(y)dy\\
&=& -i^k \rho(S) \int_0^\infty (\delta_k f)(iy) (\varphi_{2-k, \chi^{-1}}S)(y)dy\\
&=&-i^k \rho(S) L_{\delta_k f}(\varphi|_{2-k, \chi^{-1}}S).
\end{eqnarray*}
\end{proof}

Recall that $\varphi_s(x) = \varphi(x) x^{s-1}$.
We define the series $L(s,f,\varphi)$ by
\[
L(s,f,\varphi) := L_f(\varphi_s).
\]
Then, we prove that the series $L(s,f,\varphi)$ has an analytic continuation and satisfies a functional equation. 

\begin{thm} \label{harmonicvarphi_s}
Let $f\in H_{k, \chi,\rho}$, and  $n_0\in\mathbb{N}$ be such that $f_j(\tau)$ are $O(e^{2\pi n_0 v})$ as $v=\mathrm{Im}(\tau)\to\infty$ for each $1\leq j\leq m$. 
Suppose that $\varphi \in \mathcal{C}(\mathbb{R}, \mathbb{C})$ is a non-zero function such that, for some $\epsilon>0$, $\varphi(x)$ and $\varphi(x^{-1})$ are $o(e^{-2\pi(n_0+\epsilon)x})$ as $x\to\infty$. 
 Then the series 
$
L(s,f,\varphi) 
$
converges absolutely for $\mathrm{Re} > \frac12$, has an analytic continuation to all $s\in\mathbb{C}$ and satisfies the functional equation
\[
L(s,f,\varphi) = i^k \rho(S) L(1-s, f, \varphi|_{1-k,\chi^{-1}} S).
\]
\end{thm}

\begin{proof}
By the growth of $\varphi$, we see that $\mathcal{L}(|\varphi|^2)(y)$ converges absolutely for $y\geq -2\pi n_0$. 
For $y>0$ and $s\in\mathbb{C}$ with $\mathrm{Re}(s)>\frac12$, Cauchy-Schwarz inequality implies that
\[
(\mathcal{L}|\varphi_s|)(y) \leq (\mathcal{L}(|\varphi|^2)(y))^{\frac12} y^{-\mathrm{Re}(s)+\frac12} (\Gamma(2\mathrm{Re}(s)-1))^{\frac12}.
\]
Therefore, $\varphi_s\in \mathcal{F}_f$ for $\mathrm{Re}(s)>\frac12$. 

Recall that 
\[
f(iy) = (iy)^{-k} \chi^{-1}(S) \rho(S)^{-1} f\left(i\frac1y\right)
\]
for $y>0$. 
Therefore, we have
\begin{eqnarray*}
L(s, f, \varphi) &=& L_f(\varphi_s) = \int_0^\infty f(iy)\varphi_s(y)dy\\
&=& \int_0^\infty f(iy) \varphi(s) y^{s-1} dy\\
&=& \int_1^\infty f(iy) \varphi(y)y^{s-1} dy + \int_0^1 f(iy)\varphi(y)y^{s-1}dy\\
&=& \int_1^\infty f(iy) \varphi(y)y^{s-1} dy + \int_1^\infty f\left( i\frac1y \right) \varphi \left( \frac 1y \right) y^{1-s}y^{-2}dy\\
&=& \int_1^\infty f(iy) \varphi(y)y^{s-1} dy + i^k \rho(S) \int_1^\infty f(iy) \left(\varphi|_{1-k,\chi^{-1}}S\right)(y) y^{-s}dy.
\end{eqnarray*}
By the growth of $\varphi$ at $0$ and $\infty$, we see that the integrals 
\[
\int_1^\infty f(iy) \varphi(y)y^{s-1} dy\  \text{and}\ \int_1^\infty f(iy) \left(\varphi|_{1-k,\chi^{-1}}S\right)(y) y^{-s}dy
\]
 are well-defined for all $s\in\mathbb{C}$, and give holomorphic functions.

Since $f\in H_{k, \chi,\rho}$, we obtain that $\rho(-I_2) \chi(-I_2) = - I_m$, where $I_m$ denotes the identity matrix of size $m$. 
Therefore, we get the desired functional equation. 
\end{proof}

We now state the converse theorem in the case of vector-valued harmonic weak Maass forms.

\begin{thm} \label{converse}
For each $1\leq j\leq m$, let $(c^+_{f,j}(n))_{n\geq-n_0}$ and $(c^-_{f,j}(n))_{n+\kappa_j<0}$ be sequences of complex numbers such that $c^+_{f,j}(n), c^-_{f,j}(n) = O(e^{C\sqrt{|n|}})$ as $|n|\to\infty$, for some $C>0$. 
For each $\tau\in\mathbb{H}$, set
\[
f(\tau) := \sum_{j=1}^m \sum_{n=-n_0}^\infty c^+_{f,j}(n)e^{2\pi i(n+\kappa_j)\tau}\mbf{e}_j +  \sum_{j=1}^m \sum_{n+\kappa_j<0}c^-_{f,j}(n) \Gamma(1-k, -4\pi (n+\kappa_j)v)e^{2\pi i(n+\kappa_j)\tau} \mbf{e}_j.
\]
Suppose that for each $\varphi\in S_c(\mathbb{R}_+)$, the functions $L_f(\varphi)$ and $L_{\delta_k f}(\varphi)$ satisfy 
\[
L_f(\varphi) = i^k \rho(S) L_f(\varphi|_{2-k, \chi^{-1}} S)
\]
and
\[
L_{\delta_k f}(\varphi) = -i^k \rho(S) L_{\delta_k f}(\varphi|_{2-k, \chi^{-1}}S).
\]
Then, $f$ is a vector-valued harmonic weak Maass form in $H_{k, \chi,\rho}$.
\end{thm}

\begin{proof}
By the bounds for $c^+_{f,j}(n)$ and $c^-_{f,j}(n)$, we see that $f_j(\tau)$ and $(\delta_k(f))_j$ converge absolutely to  smooth functions on $\mathbb{H}$ for $1\leq j\leq m$. 
Note that $\varphi_s \in S_c(\mathbb{R}_+)$ for any $s\in\mathbb{C}$ and $\varphi\in S_c(\mathbb{R}_+)$.
Since $\varphi\in S_c(\mathbb{R}_+)$, there exist $0<c_1<c_2$ and $C>0$ such that $\mathrm{Supp}(\varphi) \subset [c_1, c_2]$ and $|\varphi(y)|\leq C$ for any $y>0$. 
Then, for each $1\leq j\leq m$ and $n+\kappa_j>0$, we have
\begin{eqnarray*}
|c^+_{f,j}(n)| (\mathcal{L}|\varphi_s|) (2\pi (n+\kappa_j)) &\leq& C |c^+_{f,j}(n)| e^{-2\pi (n+\kappa_j)c_1} (c_2-c_1) \max\{c_1^{\mathrm{Re}(s)-1}, c_2^{\mathrm{Re}(s)-1}\}.
\end{eqnarray*}
This implies that for each $1\leq j\leq m$, we have
\begin{eqnarray*}
\sum_{n\in\mathbb{Z} \atop n+\kappa_j>0} |c^+_{f,j}(n)| (\mathcal{L}|\varphi_s|)(2\pi (n+\kappa_j)) \leq C(c_2-c_1)\max\{c_1^{\mathrm{Re}(s)-1}, c_2^{\mathrm{Re}(s)-1}\} \sum_{n\in\mathbb{Z} \atop n+\kappa_j>0} |c^+_{f,j}(n)| e^{-2\pi (n+\kappa_j) c_1}.
\end{eqnarray*}

For each $1\leq j\leq m$ and $n+\kappa_j<0$, we have
\begin{eqnarray*}
(\mathcal{L}|\varphi_{s+1-k}|)(-2\pi (n+\kappa_j)(2t+1)) &\ll& \int_{c_1}^{c_2} e^{2\pi(n+\kappa_j)(2t+1)} y^{\mathrm{Re}(s)-k} dy\\
&\ll& e^{2\pi(n+\kappa_j)c_1(2t+1)}\max\{c_1^{\mathrm{Re}(s)-k}, c_2^{\mathrm{Re}(s)-k}\}
\end{eqnarray*}
This implies that 
\begin{eqnarray*}
&&\sum_{n+\kappa_j<0} |c_{f,j}^-(n)| |4\pi (n+\kappa_j)|^{1-k} \int_0^\infty \frac{(\mathcal{L}|\varphi_{s+1-k})(-2\pi (n+\kappa_j)(2t+1)}{(1+t)^k} dt\\
&&\ll \max\{c_1^{\mathrm{Re}(s)-k}, c_2^{\mathrm{Re}(s)-k} |4\pi (n+\kappa_j)|^{1-k}\int_0^\infty \frac{e^{-4\pi t c_1 n_1}}{(1+t)^k}dt \sum_{n+\kappa_j<0} e^{2\pi (n+\kappa_j)c_1} |c_{f,j}^-(n)|,
\end{eqnarray*}
where $n_1$ is defined by
\[
n_1 :=
\begin{cases}
    1 & \text{if $\kappa_j = 0$},\\
    \kappa_j & \text{if $\kappa_j \neq0$}.
\end{cases}
\]
Therefore, $\varphi_s \in \mathcal{F}_f$, and $L_f(\varphi_s)$ is an analytic function on $s\in\mathbb{C}$ by Weierstrass theorem. 
In the same way, we see that $L_{\delta_k f}(\varphi_s)$ is an analytic function for $s\in \mathbb{C}$.

Recall that by Theorem \ref{integralrepresentation} we have
\[
L_f(\varphi_s) = \int_0^\infty f(iy) \varphi_s(y) dy
\]
and
\[
L_{\delta_k f}(\varphi) = \int_0^\infty (\delta_k f)(iy) \varphi(y) dy.
\]
By the Mellin inversion formula, we have
\[
f(iy) \varphi(y) = \frac{1}{2\pi i}\int_{c-i\infty}^{c+i\infty} L_f(\varphi_s) y^{-s} ds
\]
and
\[
(\delta_k f)(iy) \varphi(y) = \frac{1}{2\pi i} \int_{c-i\infty}^{c+i\infty} L_{\delta_k f}(\varphi_s) y^{-s}ds
\]
for all $c\in \mathbb{R}$ and $y>0$.

By integration by parts, we see that for each $1\leq j\leq m$, 
\[
|\langle L_f(\varphi_s), \mbf{e}_j\rangle | \leq \frac{1}{|s|} \int_0^\infty \left|\frac{d}{dy} (f_j(iy)\varphi(y))\right| y^{\mathrm{Re}(s)} dy.
\]
Therefore, $L_f(\varphi_s) \to 0$ as $\mathrm{Im}(s)\to \infty$. 
The corresponding fact for $L_{\delta_k f}(\varphi_s)$ can be proved in the same way.
From this, we have
\begin{eqnarray*}
f(iy) \varphi(y) &=& \frac{1}{2\pi i}\int_{c-i\infty}^{c+i\infty} L_f(\varphi_s) y^{-s}ds=\frac{1}{2\pi i} \int_{k-c-i\infty}^{k-c+i\infty} L_f(\varphi_s) y^{-s}ds\\
&=& \frac{1}{2\pi i} \int_{c-i\infty}^{c+i\infty} L_f(\varphi_{k-s}) y^{-k+s} ds,
\end{eqnarray*}
and by Theorem \ref{ftnaleqn}, we see that
\begin{equation} \label{harmonicftneqn1}
f(iy) \varphi(y) = \frac{i^k}{2\pi i} \rho(S) \int_{c-i\infty}^{c+i\infty} L_f(\varphi_{k-s}|_{2-k,\chi^{-1}} S)y^{-k+s}ds
\end{equation}
for any $y>0$.
In the same way, we have
\begin{eqnarray*}
(\delta_k f)(iy) \varphi(y) &=& \frac{1}{2\pi i} \int_{c-i\infty}^{c+i\infty} L_{\delta_k(f)}(\varphi_{k-s}) y^{-k+s}ds\\
&=& -\frac{i^k}{2\pi i} \rho(S) \int_{c-i\infty}^{c+i\infty} L_{\delta_k(f)}(\varphi_{k-s}|_{2-k, \chi^{-1}}S)y^{-k+s}ds
\end{eqnarray*}
for any $y>0$.

Since we have
\[
L_f(\varphi_{k-s}|_{2-k} S) = \int_0^\infty f(iy)(\varphi_{k-s}|_{2-k,\chi^{-1}} S)(y)dy = \int_0^\infty f(iy) \chi(S) \varphi\left( \frac 1y \right) y^{s-1} dy,
\]
the Mellin inversion formula implies that
\[
f(iy) \chi(S) \varphi\left( \frac1y \right)  = \frac{1}{2\pi i}\int_{c-i\infty}^{c+i\infty} L_f(\varphi_{k-s}|_{2-k,\chi^{-1}}S)y^{-s}dy
\]
for all $c\in \mathbb{R}$ and $y>0$.
Therefore, we have
\begin{equation} \label{harmonicftneqn2}
f\left( \frac{i}{y} \right) \chi(S) \varphi(y) = \frac{1}{2\pi i} \int_{c-i\infty}^{c+i\infty} L_f(\varphi_{k-s}|_{2-k,\chi^{-1}}S) y^{s} dy,
\end{equation}
and from (\ref{harmonicftneqn1}) and (\ref{harmonicftneqn2}), we see that 
\[
f(iy) \varphi(y) = i^k y^{-k} \rho(S) \chi(S) f\left( \frac iy \right) \varphi(y)
\]
for any $y>0$.
Since for each $y>0$, there exists $\varphi \in S_c(\mathbb{R}_+)$ such $\varphi(y) \neq 0$, we have
\begin{equation} \label{harmonicftneqn3}
 f(iy)= i^k y^{-k} \rho(S) \chi(S) f\left( \frac iy \right)
\end{equation}
for any $y>0$. 
In the same way, we see that 
\begin{eqnarray*}
L_{\delta_k f}(\varphi_{k-2}|_{2-k, \chi^{-1}}S) = \int_0^\infty (\delta_k f)(iy) \chi(S) \varphi\left( \frac 1y\right) y^{s-1}dy,    
\end{eqnarray*}
and hence we have
\begin{eqnarray*}
(\delta_k f)(iy) \left(i\frac 1y \right) \chi(S) \varphi(y) = \frac{1}{2\pi i} \int_{c-i\infty}^{c+i\infty} L_{\delta_k f}(\varphi_{k-s}|_{2-k,\chi^{-1}}S)y^s dy.
\end{eqnarray*}
Therefore, we obtain
\begin{eqnarray} \label{harmonicftneqn4}
(\delta_k f)(iy) = -i^k y^{-k} \rho(S) \chi(S) (\delta_k f)\left( i\frac 1y\right).
\end{eqnarray}

We now define 
\[
F(\tau) = f(\tau) - (f|_{k,\chi,rho}S)(\tau)
\]
for $\tau\in \mathbb{H}$.
By (\ref{harmonicftneqn3}) and (\ref{harmonicftneqn4}), we see that $F(iv) = 0$ and $\frac{\partial}{\partial u} F(iv) = 0$.
This implies that $F\equiv 0$
since $F$ is an eigenfunction of the Laplace operator.
Therefore, 
\[
f = f|_{k, \chi,\rho} S.
\]
Since $T$ and $S$ generates $\mathrm{SL}_2(\mathbb{Z})$, we see that $f$ is a harmonic weak Maass  form in $H_{k,\chi,\rho}$. 
\end{proof}

We define the operator $\xi_{2-k}$ by
\[
\xi_{2-k}:=2iv^{2-k}\overline{\frac{\partial}{\partial \overline{\tau}}}.
\]
Then, it defines a surjective map  \cite[Proposition 3.2]{BF}
\begin{align*}
\xi_{2-k}: H_{2-k,\chi,\rho} \to S_{k,\overline{\chi},\overline{\rho}}.
\end{align*}
Note that if $f\in S_{k,\overline{\chi}, \overline{rho}}$, then $f$ has a Fourier expansion of the form
\begin{equation} \label{vvmfcusp}
f(\tau) = \sum_{j=1}^m \sum_{n-\kappa_j>0} a_{f,j}(n)e^{2\pi i(n-\kappa_j)\tau}\mbf{e}_j.
\end{equation}
Let  $\mathcal{C}_c^\infty (\mathbb{R}, \mathbb{R})$ be the space of piecewise smooth, compactly supported functions on $\mathbb{R}$ with values in $\mathbb{R}$.
Then, we have the following summation formula for harmonic weak Maass forms.

\begin{thm}
Let $k\in 2\mathbb{N}$ and let $f\in S_{k, \bar{\chi}, \bar{\rho}}$ with Fourier expansion as in (\ref{vvmfcusp}).
Suppose that $g$ is an element of $H_{2-k, \chi,\rho}$ such that $\xi_{2-k}(g) = f$ with Fourier expansion
\[
g(\tau) = \sum_{j=1}^m \sum_{n\gg-\infty} c_{g,j}^+(n) e^{2\pi i(n+\kappa_j)\tau} \mbf{e}_j + \sum_{j=1}^m \sum_{n+\kappa_j<0} c_{g,j}^-(n) \Gamma(k-1, -4\pi (n+\kappa_j) v) e^{2\pi i(n+\kappa_j)\tau} \mbf{e}_j. 
\]
For every $\varphi \in \mathcal{C}_c^\infty$ and $1\leq j\leq m$, we have
\begin{eqnarray*}
&&\sum_{n\gg-\infty} c_{g,j}^+(n)\int_0^\infty \varphi(y)\left( e^{-2\pi (n+\kappa_j)y} -(-iy)^{k-2}e^{-2\pi(n+\kappa_j)/y}\right)dy   \\
&&= \sum_{l=0}^{k-2} \sum_{n-\kappa_j>0} \overline{a_{f,j}(n)}\bigg( \frac{(k-2)!}{l!} (4\pi (n-\kappa_j)^{1-k+l} \int_0^\infty e^{-2\pi(n-\kappa_j)y}y^l \varphi(y) dy \\
&&+ \frac{2^{l+1}}{(k-1)} (8\pi (n-\kappa_j))^{-\frac{k}{2}} \int_0^\infty e^{-\pi (n-\kappa_j)y}y^{\frac k2 - 1}\varphi(y) M_{1-\frac k2 + l, \frac{k-1}2}(2\pi (n-\kappa_j)y)dy \bigg),
\end{eqnarray*}
where $M_{\kappa, \mu}(z)$ denotes the Whittaker hypergeometric function (see \cite[Section 13.14]{OLBC}).
\end{thm}

\begin{proof}
Note that $\mathcal{C}_c^\infty(\mathbb{R}, \mathbb{R}) \subset \mathcal{F}_f \cap \mathcal{F}_g$.
Since $a_{f,j}(n) = -\overline{c_{g,j}^-(n)}(4\pi (n-\kappa_j))^{k-1}$ for $n-\kappa_j>0$, by Theorem \ref{integralrepresentation}, we have
\begin{eqnarray*}
L_g(\varphi) &=& \int_0^\infty g^+(iy)\varphi(y)dy  - \overline{L_f(\Phi(\varphi))},
\end{eqnarray*}
where $\Phi(\varphi)$ is defined by
\[
\Phi(\varphi):= \mathcal{L}^{-1}\left( (2u)^{1-k}\int_0^\infty \Gamma(k-1, 2uy)e^{uy}\varphi(y)dy \right).
\]

Recall that $L_g(\varphi)$ satisfies a functional equation
\[
L_g(\varphi) = i^{2-k} \rho(S) L_g(\varphi|_{k,\chi^{-1}}S).
\]
Therefore, we have
\[
\int_0^\infty g^+(iy)\varphi(y)dy  - i^{2-k}\rho(S)  \int_0^\infty g^+(iy) (\varphi|_{k, \chi^{-1}}S)(y)dy = \overline{L_f(\Phi(\varphi))} -i^{2-k}\rho(S) \overline{L_f(\Phi(\varphi|_{k,\chi^{-1}}S)}).
\]
Note that by the functional equation of $L_f(\varphi)$, we have
\begin{eqnarray*}
i^{2-k}\rho(S) \overline{L_f(\Phi(\varphi|_{k,\chi^{-1}}S))} &=& i^{2-k}\rho(S) \overline{i^k \overline{\rho(S)} L_f(\Phi(\varphi|_{k,\chi^{-1}}S)|_{2-k, \chi}S)}\\
&=& -\rho(-I) \overline{ L_f(\Phi(\varphi|_{k,\chi^{-1}}S)|_{2-k, \chi}S)}.
\end{eqnarray*}

By \cite[4.1 (25)]{EMOT}, we have
\begin{eqnarray*}
\mathcal{L}\left(u^{v-1} f\left( \frac 1u\right)\right)(x) = x^{-\frac12 v}\int_0^\infty u^{\frac12 v}J_v\left(2u^{\frac12}x^{\frac12}\right) \mathcal{L}(f)(u)du
\end{eqnarray*}
for $\mathrm{Re}(v)>-1$, where $J_v(z)$ denotes the Bessel function defined in \cite[Section 10.2.2]{OLBC}
Therefore, we see that
\begin{eqnarray} \label{sumformula1}
\nonumber &&\mathcal{L}(\Phi(\varphi|_{k,\chi^{-1}}S)|_{2-k, \chi^{-1}} S)(2\pi (n-\kappa_j))\\
\nonumber &&= \mathcal{L}\left(\chi^{-1}(S) x^{k-2} \mathcal{L}^{-1}\left((2u)^{1-k} \int_0^\infty \Gamma(k-1, 2uy) e^{uy}\varphi\left( \frac 1y \right) \chi(S) y^{-k} dy \right) \left( \frac 1x \right)\right)(2\pi (n-\kappa_j)\\
\nonumber &&= (2\pi (n-\kappa_j))^{-\frac12(k-1)} \chi^{-1}(S) \int_0^\infty u^{\frac12 (k-1)} J_{k-1}\left(2u^{\frac12}(2\pi (n-\kappa_j))^{\frac12}\right)\\
&&\qquad \times (2u)^{1-k}\int_0^\infty \Gamma(k-1, 2uy) e^{uy}\varphi\left(\frac 1y \right) \chi(S) y^{-k}dydu\\
\nonumber &&= (2\pi (n-\kappa_j))^{-\frac12(k-1)} \chi^{-1}(S) \int_0^\infty u^{\frac12 (k-1)} J_{k-1}\left(2u^{\frac12}(2\pi (n-\kappa_j))^{\frac12}\right)\\
\nonumber &&\qquad \times (2u)^{1-k}\int_0^\infty \Gamma(k-1, 2u/y) e^{u/y}\varphi\left( y \right) \chi(S) y^{k-2}dydu\\
\nonumber &&= (8\pi (n-\kappa_j))^{\frac 12 (1-k)} \int_0^\infty \varphi(y)y^{k-2} \int_0^\infty u^{\frac12 (1-k)} J_{k-1}\left(\sqrt{8\pi (n-\kappa_j) u}\right) \Gamma(k-1, 2u/y)e^{u/y} dudy.
\end{eqnarray}
Note that by \cite[(8.4.8)]{OLBC} we have
\[
\Gamma(n+1, z) = n! e^{-z} \sum_{l=0}^n \frac{z^l}{l!}.
\]
Therefore, (\ref{sumformula1}) is equal to 
\begin{eqnarray*}
&&(8\pi (n-\kappa_j))^{\frac12 (1-k)} (k-2)!\sum_{l=0}^{k-2} \frac{2^l}{l!} \int_0^\infty \varphi(y)y^{k-2-l} \int_0^\infty u^{\frac12 (1-k)+l} J_{k-1}\left(\sqrt{8\pi (n-\kappa_j) u}\right) e^{-u/y} dudy\\
&& = (8\pi (n-\kappa_j))^{\frac12 (1-k)} (k-2)!\\
&&\qquad \times \sum_{l=0}^{k-2} \frac{2^{l+1}}{l!} \int_0^\infty \varphi(y)y^{k-2-l} \int_0^\infty u^{2-k+2l} J_{k-1}\left(\sqrt{8\pi (n-\kappa_j) } u \right) e^{-u^2/y} dudy
\end{eqnarray*}
By \cite[(8.4.8)]{EMOT}, we have
\[
\int_0^\infty e^{-\beta^2 x^2} J_v(ax) x^{s-1}dx = \frac{\Gamma\left(\frac 12v + \frac12 s\right)}{a\beta^{s-1}\Gamma(v+1)} e^{-\frac{a^2}{8\beta^2}} M_{\frac12 s-\frac12, \frac12 v}\left(\frac{a^2}{4\beta^2} \right)
\]
for $|\mathrm{arg} \beta|<\frac{\pi}{4}$ and $\mathrm{Re}(s) > -\mathrm{Re}(v)$.
Therefore, (\ref{sumformula1}) is equal to 
\begin{eqnarray*}
\frac{(8\pi (n-\kappa_j))^{\frac {-k}{2}}}{k-1} \sum_{l=0}^{k-2} 2^{l+1} \int_0^\infty \varphi(y)y^{\frac k2-1} e^{-\pi (n-\kappa_j)y} M_{1+l-\frac k2, \frac12 (k-1)}(2\pi (n-\kappa_j)y) dy.
\end{eqnarray*}
We can also see that 
\begin{eqnarray*}
\mathcal{L}(\Phi(\varphi))(2\pi(n-\kappa_j)) &=& (4\pi(n-\kappa_j))^{1-k} \int_0^\infty \Gamma(k-1, 4\pi(n-\kappa_j)y)e^{2\pi(n-\kappa_j)y} \varphi(y)dy\\
&=& \sum_{l=0}^{k-2} \frac{(k-2)!}{l!} (4\pi (n-\kappa_j))^{1-k+l} \int_0^\infty e^{-2\pi(n+\kappa_j)y} y^l \varphi(y) dy.
\end{eqnarray*}
\end{proof}

\section{$L$-series of harmonic Maass  Jacobi forms}
In this section, we consider the case of Jacobi forms and prove a converse theorem as well. We review basic notions of Jacobi forms (for more details, see  \cite[Section 3.1]{CL2} and \cite[Section 5]{EZ}).
\par
Let $k$ be a positive even integer and $m$ be a positive integer. 
From now on, we use the notation $\tau = u+iv\in\HH$ and $z = x+iy\in\CC$.

Let $F$ be a complex-valued function on $\mathbb{H}\times \mathbb{C}$.
For $\gamma=\sm a&b\\c&d\esm \in\mathrm{SL}_2(\mathbb{Z}) , X = (\lambda, \mu)\in \mathbb{Z}^2$, we define
\[(F|_{k,m} \gamma)(\tau,z) := (c\tau+d)^{-k}e^{-2\pi im\frac{cz^2}{c\tau+d}}F(\gamma(\tau,z))\]
and
\[(F|_m X)(\tau,z) :=e^{2\pi i m (\lambda^2 \tau + 2\lambda z)}F(\tau,z+\lambda\tau+\mu),\]
where $\gamma(\tau,z) = (\frac{a\tau+b}{c\tau+d},\frac{z}{c\tau+d})$.

We now define a Jacobi form.

\begin{dfn}
A weakly holomorphic Jacobi form of weight $k$ and index $m$ on $\mathrm{SL}_2(\mathbb{Z})$  is a holomorphic function $F$ on $\mathbb{H}\times\mathbb{C}$ satisfying
\begin{enumerate}
\item[(1)] $F|_{k,m} \gamma =F$ for every $\gamma\in\mathrm{SL}_2(\mathbb{Z})$,
\item[(2)] $F|_m X = F$ for every $X\in \mathbb{Z}^2$,
\item[(3)] $F$ has the Fourier expansion of the form
\begin{equation} \label{Jacobifourier}
F(\tau,z) =
\sum_{\substack{l, r\in\mathbb{Z}\\ 4ml - r^2 \gg -\infty}}a_F(l,r)e^{2\pi il\tau}e^{2\pi irz}.
\end{equation}
\end{enumerate}
\end{dfn}

We denote by $J^!_{k,m}$ the space of all weakly holomorphic Jacobi forms of weight $k$ and index $m$ on $\mathrm{SL}_2(\mathbb{Z})$. 
If a Jacobi form satisfies the condition $a(l,r)\neq 0$ only if $4ml - r^2\geq0$ (resp. $4ml-r^2>0$), then it is called a Jacobi form (resp. Jacobi cusp form). 
We denote by $J_{k,m}$ (resp. $S_{k,m}$) the space of all Jacobi forms (resp., Jacobi cusp forms) of weight $k$ and index $m$ on $\mathrm{SL}_2(\mathbb{Z})$.

We now recall some definitions and facts about harmonic Maass-Jacobi forms, which were interested in \cite{BRR} and \cite{BR}.
The Casimir operators $C_{k,m}$ is defined  by
\begin{multline*}
    \mathcal{C}_{k,m}:=-2(\tau-\overline{\tau})^2\partial_{\tau \overline{\tau}}-(2k-1)(\tau-\overline{\tau})\partial_{\overline{\tau}}+\frac{(\tau-\overline{\tau})^2}{4\pi i m}\partial_{\tau z z}+\frac{k(\tau-\overline{\tau})}{4\pi i m}\partial_{z \overline{z}}\\
    +\frac{(\tau-\overline{\tau})(z-\overline{z})}{4\pi i m}\partial_{zz\overline{z}}-2(\tau-\overline{\tau})(z-\overline{z})\partial_{\tau \overline{z}}+(1-k)(z-\overline{z})\partial_{\overline{z}}\\
    +\frac{(\tau-\overline{\tau})^2}{4\pi i m}\partial_{\tau \overline{z} \overline{z}}+\left(\frac{(z-\overline{z})^2}{2}+\frac{k(\tau-\overline{\tau})}{4\pi i m} \right)\partial_{\overline{z}\overline{z}}+\frac{(\tau-\overline{\tau})(z-\overline{z}) }{4\pi i m}\partial_{z\overline{z}\overline{z}}.
\end{multline*}

\begin{dfn}\label{hmj} 
Let $F: \mathbb{H}\times \mathbb{C}\to \mathbb{C}$ be a function that is real-analytic in $\tau\in\mathbb{H}$ and holomorphic in $z\in\mathbb{C}$.
Then,  $F$ is called a harmonic Maass-Jacobi form 
of weight $k$ and index $m$ on $\mathrm{SL}_2(\mathbb{Z})$ if the following conditions are satisfied:
\begin{enumerate}
\item $F|_{k,m} \gamma =F$ for every $\gamma \in \mathrm{SL}_2(\mathbb{Z})$.
\item $F|_m X = F$ for every $X\in\mathbb{Z}^2$
\item $\mathcal{C}_{k,m}(\phi)=0$.
\item There is a function 
\[
P_{F}(\tau,z) =\sum_{l,r\in\mathbb{Z}\atop 4ml-r^2 \leq 0} c^+_{F}(4ml-r^2) e^{2\pi il\tau} e^{2\pi irz}
\] 
such that there are only finitely many $4ml-r^2\leq 0$ such that $c^+_{F}(4ml-r^2)\neq 0$ and 
$F(\tau,z)-P_{F}(\tau,z)=O(e^{-hv}e^{2\pi my^2/v})$ as $v\rightarrow\infty$ for some $h>0$.
\end{enumerate}
\end{dfn}

We use $\hat{J}_{k,m}$ to denote the space of such forms. 
Any $F\in \hat{J}_{k,m}$ has a Fourier expansion  of the form
\begin{align}\label{JacobiFourier}
F(\tau,z) 
\nonumber&=\sum_{l,r \in\mathbb{Z}\atop 4ml-r^2\gg -\infty} c^+_{F}(l,r)e^{2\pi il\tau} e^{2\pi irz}+\sum_{l,r \in\mathbb{Z}\atop 4ml-r^2>0} c^-_{F}(l,r)\Gamma\left(1-k, \frac{(4ml-r^2)\pi v}{m}\right)e^{2\pi il\tau} e^{2\pi irz}.
\end{align}

A harmonic Maass-Jacobi form has the theta decomposition (see \cite[Section 5]{BRR}, \cite[Section 6]{BR}, and \cite[Section 5]{EZ})
\begin{equation} \label{thetadecomp}
F(\tau,z)=\sum_{j=1}^{2m}F_j(\tau)\theta_{m,j}(\tau,z),
\end{equation}
where $\theta_{m,j}$ is defined by
\[
\theta_{m,j}(\tau,z):=\sum_{r\equiv j \pmod{2m}} e^{\pi ir^2/(2m)} e^{2\pi irz}
\]
and  $\chi_\eta$ is the eta-multiplier system defined by 
\[
\chi_\eta(\gamma) :=  \frac{\eta(\gamma\tau)}{\sqrt{c\tau+d}\ \eta(\tau)}
\]
for $\gamma = \sm a&b\\c&d\esm \in \mathrm{SL}_2(\mathbb{Z})$.
Moreover, $\sum_{j=1}^{2m} F_j\mbf{e}_j$ is a vector-valued harmonic Maass form. To explain this, we recall the definition of metaplectic groups.
The metaplectic group $\mathrm{Mp}_2(\mathbb{R})$ 
consists of pairs $(g,\omega(\tau))$, where $g=\sm a&b\\c&d\esm \in \mathrm{SL}_2(\mathbb{R})$ and $\omega: \mathbb{H}\to \mathbb{C}$ is a holomorphic function satisfying $\omega(\tau)^2= c\tau+d$, with group law 
\[
(g,\omega(\tau))(g',\omega'(\tau)):=(gg', (\omega\circ g')(\tau)\omega'(\tau)).
\]
We use $\mathrm{Mp}_2(\mathbb{Z})$ to denote the inverse image of $\mathrm{SL}_2(\mathbb{Z})$ in $\mathrm{Mp}_2(\mathbb{R})$.
Let $m$ be a positive integer. We also recall the Weil representation $\rho_m: \mathrm{Mp}_2(\mathbb{Z})\to \GL_{2m}(\mathbb{C})$ given by 
\begin{align*}
&\rho_m(T)\mbf{e}_l:=e_{4m}(l^2)\mbf{e}_l,\\
&\rho_m(S)\mbf{e}_l:=\frac{1}{\sqrt{2im}}\sum_{l'=1}^{2m}e_{2m}(-ll')\mbf{e}_{l'},
\end{align*}
where we use the notation $e_{m}(w):=e^{\frac{2\pi i w}{m}}$.
Let $T:=\big(\sm 1&1\\0&1\esm ,1\big)$ and $S:=\big(\sm 0&{-1}\\1&0\esm, \sqrt{\tau}\big)$ be two generators of $\mathrm{Mp}_2(\mathbb{Z})$.
Throughout this paper, we use the convention that $\sqrt{\tau}$ is chosen so that 
  $\arg(\sqrt{\tau})\in (-\pi/2, \pi/2]$.
  The map
\[\sm a&b\\c&d\esm  \mapsto \widetilde{\sm a&b\\c&d\esm} = (\sm a&b\\c&d\esm, \sqrt{c\tau+d})\]
defines a locally isomorphic embedding of $\mathrm{SL}_2(\mathbb{R})$ into $\mathrm{Mp}_2(\mathbb{R})$. 
We then define a representation $\rho'_{m} : \mathrm{SL}_2(\mathbb{Z}) \to \mathrm{GL}_{2m}(\mathbb{C})$ by
\[\rho_{m}'(\gamma) := \rho_m(\widetilde{\gamma})\chi_\eta(\gamma)\]
for $\gamma\in\mathrm{SL}_2(\mathbb{Z})$.
It is known (\cite{CL2}, p.281) that $\rho_m'$ is unitary representation of $\mathrm{SL}_2(\mathbb{Z})$.
The theta decomposition induces an isomorphism $\phi_{k,m}$ between $H_{k-\frac12,\overline{\rho_m},\overline{\chi_\eta}}$ and $\hat{J}_{k,m}$
(see \cite[Section 5]{BRR} and \cite{CC}): 
\[
F \mapsto \sum_{j=1}^{2m} F_j(\tau)\theta_{m,j}(\tau,z).
\]

Let $F$ be a Jacobi cusp form $F\in S_{k,m}$ with its Fourier expansion
\[
F(\tau,z) =
\sum_{\substack{l, r\in\mathbb{Z}\\ 4ml - r^2 >0}}a_F(l,r)e^{2\pi il\tau}e^{2\pi irz}.
\]
Then, $F_j$ has the Fourier expansion
\[
F_j(\tau) =\sum_{\substack{n>0\\ n+j^2 \equiv 0 \pmod{4m}}} a_F\left(\frac{n+j^2}{4m}, j\right)e^{2\pi in\tau/(4m)}.
\]
We define the partial $L$-series of $F$ by
\[
L(F,j,s) := \sum_{\substack{n>0\\  n+j^2 \equiv 0 \pmod{4m}}} \frac{a_F\left(\frac{n+j^2}{4m}, j\right)} {\left(\frac{n}{4m}\right)^{s}}
\]
for $1\leq j\leq 2m$.
This $L$-series was studied in \cite{Ber, CL2, LR, LR2}.

We now consider $L$-series of harmonic Maass Jacobi forms.
Let $F\in \hat{J}_{k,m}$.
Then, it has the theta decomposition as in (\ref{thetadecomp}), and $F_j$ has the Fourier expansion
\begin{eqnarray*}
F_j(\tau) &=& \sum_{n\gg-\infty \atop n+j^2 \equiv 0\pmod{4m}} c^+_{F}\left(\frac{n+j^2}{4m}, j\right) e^{2\pi in\tau/(4m)}\\
&&+ \sum_{n<0\atop n+j^2 \equiv 0 \pmod{4m}} c^-_F\left(\frac{n+j^2}{4m}, j\right) \Gamma\left(\frac32-k, -\frac{-\pi nv}{m}\right) e^{2\pi in\tau/(4m)}.
\end{eqnarray*}
Let $n_0\in\mathbb{N}$ be such that $F_j(\tau)$ are $O(e^{2\pi n_0 v})$ as $v=\mathrm{Im}(\tau)\to\infty$ for each $1\leq j\leq m$.
Let $\mathcal{F}_F$ be the space of functions $\varphi\in \mathcal{C}(\mathbb{R}, \mathbb{C})$ such that 
\begin{enumerate}
\item $(\mathcal{L} \varphi)(s)$ converges absolutely for all $s$ with $\mathrm{Re}(s) \geq -2\pi n_0$,

\item $(\mathcal \varphi_{\frac52-k})(s)$ converges absolutely for all $s$ with $\mathrm{Re}(s)>0$,

\item for each $1\leq j\leq m$, the series
\begin{eqnarray*}
&& \sum_{n\gg-\infty \atop n+j^2 \equiv 0 \pmod{4m}} \left|c^+_{F}\left(\frac{n+j^2}{4m}, j\right) \right| (\mathcal{L} |\varphi|)(\pi n/(2m))\\
&& + \sum_{n<0\atop n+j^2 \equiv 0 \pmod{4m}} \left|c^-_{F}\left(\frac{n+j^2}{4m}, j\right) \right| (\pi |n+\kappa_j|/m)^{\frac32-k} \int_0^\infty \frac{(\mathcal{L}|\varphi_{\frac52-k}|)(-\pi n(2t+1)/(2m))}{(1+t)^{k-\frac12}} dt
\end{eqnarray*}
converges.
\end{enumerate}
For $\varphi\in \mathcal{F}_F$, we define the $L$-series of $F$ by
\begin{eqnarray*}
L_F(\varphi) &:=& \sum_{j=1}^{2m} \sum_{n\gg-\infty\atop n+j^2\equiv 0 \pmod{4m}} c^+_{F}\left(\frac{n+j^2}{4m}, j\right)  (\mathcal{L}\varphi) (\pi (n+\kappa_j)/(2m))\ \mbf{e}_j\\
&& + \sum_{j=1}^{2m} \sum_{n<0\atop n+j^2 \equiv 0 \pmod{4m}} c^-_{F}\left(\frac{n+j^2}{4m}, j\right) (-\pi (n+\kappa_j)/m)^{\frac32-k} \\
&& \qquad \qquad \times \int_0^\infty \frac{(\mathcal{L}\varphi_{\frac52-k})(-\pi n(2t+1)/(2m))}{(1+t)^{k-\frac12}} dt\ \mbf{e}_j.
\end{eqnarray*}

Let $L_m$ be the heat operator defined by
\[
L_{m}:=\frac{2m}{\pi i}\frac{\partial}{\partial \tau}-\frac{1}{(2\pi i)^2}\frac{\partial^2}{\partial z^2}.
\]
We define a differential operator $\alpha_k$ by
\[
(\alpha_k F)(\tau,z) L= \tau \left(\frac{\partial F}{\partial \bar{\tau}} + \frac{\pi i}{2m}L_m(F)\right)(\tau,z) + \frac{2k-1}{4}F(\tau,z).
\]
Then, we prove that the corresponding vector-valued function for $\alpha_k F$ is the image of the corresponding vector-valued harmonic weak Maass form for $F$ under the operator $\delta_{k-\frac12}$.

\begin{lem}
If $F\in \hat{J}_{k,m}$, then we have
\[
\phi_{k,m}\left( \alpha_k F \right) = \delta_{k-\frac12} \left(\phi_{k,m}(F)\right).
\]
\end{lem}

\begin{proof}
Suppose that $F$ is in $\hat{J}_{k,m}$.
Then, it has the theta decomposition
\[
F(\tau,z) = \sum_{j=1}^{2m} F_j(\tau) \theta_{m,j}(\tau,z).
\]
Moreover, $(\alpha_k F)|_{m} X = \alpha_k F$ for all $X\in \mathbb{Z}^2$ and $\alpha_k F$ is holomorphic in $z\in\mathbb{C}$.
Therefore, $\alpha_k F$ also has the theta expansion
\[
(\alpha_k F)(\tau,z) = \sum_{j=1}^{2m} (\alpha_k F)_j (\tau) \theta_{m,j}(\tau,z).
 \]
Note that 
\[
L_m(\theta_{m,j}) = 0,\ \frac{\partial}{\partial \bar{\tau}} (\theta_{m,j}) = 0
\]
for every $1\leq j\leq 2m$.
Therefore, we have
\[
(\alpha_k F)_j = \delta_{k-\frac12}(F_j)
\]
for each $1\leq j\leq 2m$.
\end{proof}

Suppose that $F\in \hat{J}_{k,m}$.
Let $\mathcal{F}_{\alpha_k F}$ be the space of functions $\varphi\in C(\mathbb{R}, \mathbb{C})$ such that 
\begin{enumerate}
\item $(\mathcal{L} \varphi)(s)$ converges absolutely for all $s$ with $\mathrm{Re}(s) \geq -2\pi n_0$,

\item $(\mathcal \varphi_{\frac72-k})(s)$ converges absolutely for all $s$ with $\mathrm{Re}(s)>0$,

\item
for each $1\leq j\leq m$ the series
\begin{eqnarray*}
&& \sum_{n\gg-\infty\atop n+j^2 \equiv 0 \pmod{4m}} \left|c^+_{F}\left(\frac{n+j^2}{4m}, j\right)(n/(4m)) \right| (\mathcal{L} |\varphi_2|)(\pi n/(2m))\\
&& + \sum_{n+\kappa_j<0 \atop n+j^2\equiv 0 \pmod{4m}} \left|c^-_{F}\left(\frac{n+j^2}{4m}, j\right)(n/(4m))\right| (\pi |n|/m)^{\frac32-k} \int_0^\infty \frac{(\mathcal{L}|\varphi_{\frac72-k}|)(-\pi n(2t+1)/(2m))}{(1+t)^{k-\frac12}} dt
\end{eqnarray*} 
converges.
\end{enumerate}
For $\varphi\in \mathcal{F}_{\alpha_k F}$, we define $L_{\alpha_k F}(\varphi)$ by
\begin{eqnarray*}
L_{\alpha_k F}(\varphi) &:=& \frac{2k-1}{4}L_F(\varphi) -2\pi  \sum_{j=1}^{2m} \sum_{n\gg-\infty\atop n+j^2\equiv 0 \pmod{4m}} c^+_{f,j}(n) (n/2m) (\mathcal{L}\varphi_2) (\pi n/(2m))\ \mbf{e}_j\\
&&\qquad  -2\pi \sum_{j=1}^{2m} \sum_{n<0\atop n+j^2\equiv 0 \pmod{4m}} c^-_{F}\left(\frac{n+j^2}{4m}, j\right)(n/(4m))(-\pi n/m)^{\frac32-k}\\
&&\qquad \qquad \times \int_0^\infty \frac{(\mathcal{L}\varphi_{\frac72-k})(-\pi n(2t+1)/(2m))}{(1+t)^{k-\frac12}} dt\ \mbf{e}_j.
\end{eqnarray*}

We prove the converse theorem in the case of harmonic  Maass Jacobi forms using a similar argument as in the proof of Theorem \ref{converse}.

\begin{thm} \label{Jacobiconverse}
Let $(c^+_{F}(l,r))_{4ml-r^2\geq-D_0}$ and $(c^-_{F}(l,r))_{4ml-r^2<0}$ be sequences of complex numbers such that $c^+_F(l,r), c^-_{F}(l,r) = O(e^{C\sqrt{|4ml-r^2|}})$ as $|4ml-r^2|\to\infty$, for some $C>0$. 
For each $\tau\in\mathbb{H}$ and $z\in\mathbb{C}$, set
\[
F(\tau,z) := \sum_{l,r \in\mathbb{Z}\atop 4ml-r^2\gg -\infty} c^+_{F}(l,r)e^{2\pi il\tau} e^{2\pi irz}+\sum_{l,r \in\mathbb{Z}\atop 4ml-r^2>0} c^-_{F}(l,r)\Gamma\left(1-k, \frac{(4ml-r^2)\pi v}{m}\right)e^{2\pi il\tau} e^{2\pi irz}.
\]
Suppose that for each $\varphi\in S_c(\mathbb{R}_+)$, the functions $L_F(\varphi)$ and $L_{\alpha_k F}(\varphi)$ satisfy 
\[
L_F(\varphi) = i^{k-\frac12} \overline{\rho_m}(S) L_f(\varphi|_{2-k+\frac12, \chi_\eta} S)
\]
and
\[
L_{\alpha_k F}(\varphi) = -i^{k-\frac12} \overline{\rho_m}(S) L_{\alpha_k f}(\varphi|_{2-k+\frac12, \chi_\eta}S).
\]
Then, $F$ is a harmonic Maass Jacobi form in $\hat{J}_{k, m}$.
\end{thm}

  \section{$L$-series of harmonic weak Maass forms in the Kohnen plus space}
  In this section, we consider the case of half-integral weight modular forms in the Kohnen plus space and prove a converse theorem for this case.
   Let $k$ be a positive even integer. 
By \cite[Theorem 5.4]{EZ}, there is an isomorphism $\psi_k$ between  $S_{k,1}$ and $S^+_{k-\frac12}$, where $S^+_{k-\frac12}$ denotes the space of cusp forms in the plus space of weight $k-\frac12$ on $\Gamma_0(4)$.

Let $f$ be a cusp form in $S^+_{k-\frac12}$ with Fourier expansion 
\[
f(\tau) = \sum_{\substack{n>0\\ n\equiv 0,3 \pmod{4}}} a_f(n)e^{2\pi in\tau}.
\]
Then, the $L$-function of $f$ is defined by
\[
L(f,s) := \sum_{\substack{n>0\\ n\equiv 0,3 \pmod{4}}} \frac{a_f(n)}{n^s}.
\]
For $1\leq j\leq 2$, let $c_j$ be defined by
\begin{equation} \label{pluscoeff}
a_{f,j}(n) :=
\begin{cases}
a_f(n) & \text{if $n\equiv -j^2 \pmod{4}$},\\
0 & \text{otherwise}.
\end{cases}
\end{equation}
Then, $a_f(n) = a_{f,1}(n) + a_{f,2}(n)$ for all $n$.
With this, we consider partial sums of $L(f,s)$ by
\[
L(f,j,s) :=  \sum_{\substack{n>0\\ n\equiv 0,3 \pmod{4}}} \frac{c_j(n)}{n^s}
\]
for $1\leq j\leq 2$.
Suppose that $F$ is a Jacobi cusp form in $S_{k,1}$.
By the theta decomposition, we have a corresponding vector-valued modular form $(F_1(\tau), F_2(\tau))$. 
Then, the isomorphism $\psi_k$ from  $S_{k,1}$  to $S^+_{k-\frac12}$ is given by 
\[
\psi_k(F)(\tau) = \sum_{j=1}^2 F_j(4\tau).
\]
 From this, we see that
\[
L(\psi_k(F),j,s) = \frac{1}{4^s} L(F,j,s). 
\]

Let $H^+_{k-\frac12}$ denote the space of harmonic weak Maass forms in the plus space space of weight $k-\frac12$ on $\Gamma_0(4)$.
Then, by \cite{CC}, we see that $\psi_k$ can be extended to an isomorphism between $\hat{J}_{k,1}$ and $H^+_{k-\frac12}$ as follows. 
Suppose that $F\in \hat{J}_{k,1}$.
Then, $F$ has the theta decomposition
\[
F(\tau,z) = \sum_{j=1}^2 F_j(\tau) \theta_{m,j}(\tau,z).
\]
Then, $\psi_k(F)(\tau) = \sum_{j=1}^2 F_j(4\tau)$.

Suppose that $f\in H^+_{k-\frac12}$ with its Fourier expansion
\begin{eqnarray*}
f(\tau) &=&    \sum_{n\gg -\infty}c^+_{f}(n) e^{2\pi i n\tau}\mbf{e}_j + \sum_{n< 0}c^-_{f}(n)\Gamma\left(\frac32-k,-4\pi nv \right) e^{2\pi i n\tau}.
\end{eqnarray*}
For $1\leq j\leq 2$,  we define $c^{\pm}_{f,j}(n)$ as in (\ref{pluscoeff}
), and 
\[
f_j(\tau) =  \sum_{n\gg -\infty}c^+_{f,j}(n) e^{2\pi i n\tau}\mbf{e}_j + \sum_{n< 0}c^-_{f,j}(n)\Gamma\left(\frac32-k,-4\pi nv \right) e^{2\pi i n\tau}.
\]
For a vector-valued function $F(\tau) := \sum_{j=1}^{2} f_j(\tau/4) \mbf{e}_j$, we define the $L$-series $L_F(\varphi)$ and $L_{\delta_{k-\frac12} F}(\varphi)$ as in (\ref{vvL}) and (\ref{vvLdelta}), respectively.

We now have the following converse theorem for half-integral weight harmonic weak Maass forms in the Kohnen plus space.  

\begin{thm} \label{plusconverse}
Let $(c^+_{f}(n))_{n\geq-n_0}$ and $(c^-_{f}(n))_{n<0}$ be sequences of complex numbers such that 
\[
c^+_f(n), c^-_{f}(n) = O\left(e^{C\sqrt{|n|}}\right)
\]
as $|n|\to\infty$, for some $C>0$. 
For each $\tau\in\mathbb{H}$, set
\[
f(\tau) :=  \sum_{n\gg -n_0}c^+_{f}(n) e^{2\pi i n\tau}\mbf{e}_j + \sum_{n< 0}c^-_{f}(n)\Gamma\left(\frac32-k,-4\pi nv \right) e^{2\pi i n\tau}.
\]
Suppose that for each $\varphi\in S_c(\mathbb{R}_+)$, the functions $L_F(\varphi)$ and $L_{\alpha_k F}(\varphi)$ satisfy 
\[
L_F(\varphi) = i^{k-\frac12} \overline{\rho_2}(S) L_f(\varphi|_{2-k+\frac12, \chi_\eta} S)
\]
and
\[
L_{\alpha_k F}(\varphi) = -i^{k-\frac12} \overline{\rho_2}(S) L_{\alpha_k f}(\varphi|_{2-k+\frac12, \chi_\eta}S).
\]
Then, $f$ is a harmonic weak Maass form in $H^+_{k-\frac12}$.
\end{thm}

\section*{Acknowledgement}
The second-named author would like to thank the FAS Dean's Office at the American University of Beirut for their support of his summer research leave.


    \end{document}